\documentclass[11pt]{amsart}

\usepackage{amsmath,amssymb,amsthm}
\usepackage{amsmath,amssymb,amsthm,amscd}
\usepackage[frame,cmtip,arrow,matrix,line,graph,curve]{xy}
\usepackage{graphpap, color}
\usepackage[mathscr]{eucal}
\usepackage{color}
\numberwithin{equation}{section}

\title{A note on small deformations of balanced manifolds}
\author{Jixiang Fu}
\author{Shing-Tung Yau}
\address{Institute of Mathematics\\ Fudan University \\ Shanghai
200433, China} \email{majxfu@fudan.edu.cn}
\address{Department of Mathematics\\ Harvard University\\ Cambridge,
MA 02138\\ USA} \email{yau@math.harvard.edu}

\newtheorem{prop}{Proposition}
\newtheorem{theo}[prop]{Theorem}

\newtheorem{coro}[prop]{Corollary}

\newtheorem{defi}[prop]{Definition}

\begin{document}
\begin{abstract}
In this note  we prove that, under a weak condition, small
deformations of a compact balanced manifold are also balanced.  This
condition is satisfied on the twistor space over a compact self-dual
four manifold.
\end{abstract}
\maketitle

\section{Introduction}
Let $(X,J)$ be a compact complex $n$-dimensional manifold. Let $g$
be a hermitian metric on $X$ and $\omega$ be the hermitian form
associated to $g$. If $d\omega^{n-1}=0$, then $g$ or $\omega$ is
called a {\sl balanced} metric. A complex manifold is called a
balanced manifold if it admits a balanced metric. The balanced
metric was first studied thoroughly by Michelsohn \cite{Mic}. She
especially proved that there exists an intrinsic characterization of
compact complex manifolds with balanced metrics by means of positive
currents.
\begin{theo} \textup{\cite{Mic}}
A compact complex manifold $X$ admits a balanced metric if and only
if any positive current $T$ on $X$ of degree $(1,1)$ must be zero if
it is the component of a boundary \textup{(}i.e., if there exists a current $S$
such that $T=\partial\bar S+\bar\partial S$\textup{)}.
\end{theo}
Using this characterization, L. Alessandrini and G. Bassanelli
proved that the existence of balanced metrics is stable under
modifications.
\begin{theo} \textup{\cite{AB2,AB3}}
Let $f:\bar X\to X$ be a proper modification of a compact complex
manifold $X$. Then $X$ admits a balanced metric if and only if $\bar
X$ admits a balanced metric.
\end{theo}

Another important  question on  balanced metrics involves the
deformation invariance. In the following, we will assume that
$\{X_t|t\in \triangle(\epsilon)\}$ is an analytic family of compact
complex manifolds. Here $\triangle(\epsilon)=\{t\in\mathbb
C||t|<\epsilon\}$. In this note, we only consider the small
deformation, that is, we may as well assume, if necessarily,
$\epsilon$ will be small enough. It is well-known \cite{MK} that any
small deformation of a compact K\"ahler manifold is K\"ahler, that
is, if $X_0$ is K\"ahler, then $X_t$ is also K\"ahler for small
enough $t$. On the other hand, the existence of a balanced metric is
not preserved under small deformations. Alessandrini and Bassanelli
observed in \cite{AB1} that the small deformation of the Iwasawa
manifold constructed by Nakamura gives such a counter-example.
However, the second author's former student C.-C. Wu in her thesis
proved that in case the complex manifold satisfies the
$\partial\bar\partial$-lemma, the balanced condition is preserved
under small deformations. Actually she proved
\begin{theo}\textup{\cite{Wu}}
If $X_0$ satisfies the $\partial\bar\partial$-lemma, then $X_t$ also
satisfies the $\partial\bar\partial$-lemma for small enough $t$.
\end{theo}
\begin{theo}\label{Wu2} \textup{\cite{Wu}} If $X_0$ satisfies the
$\partial\bar\partial$-lemma and admits a balanced metric, then
$X_t$ also admits a balanced metric for sufficiently small $t$.
\end{theo}

A complex manifold satisfies the $\partial\bar\partial$-lemma if for
its every differential form $\alpha$  such that $\partial\alpha=0$
and $\bar\partial\alpha=0$, and such that $\alpha=d\gamma$ for some
differential form $\gamma$, there is some differential form $\beta$
such that $\alpha=i\partial\bar\partial\beta$.  There are several
equivalent versions of the $\partial\bar\partial$-lemma, see
\cite{DGMS}. In this note, we consider the question of how to weaken
the $\partial\bar\partial$-lemma condition such that the existence
of balanced metrics is still preserved under small deformations.
First we give the following
\begin{defi}
A complex $n$-dimensional manifold satisfies the \textup{
$(n-1,n)$-th weak $\partial\bar\partial$-lemma} if for its every
real $(n-1,n-1)$-form $\varphi$ such that $\bar\partial\varphi$ is a
$\partial$-exact form,  there exists an $(n-2,n-1)$-form $\psi$ such
that
\begin{equation}
\bar\partial\varphi=i\partial\bar\partial\psi.
\end{equation}
\end{defi}

Certainly the condition $(n-1,n)$-th weak
$\partial\bar\partial$-lemma is weaker than the
$\partial\bar\partial$-lemma. Our main result in this note is
\begin{theo}\label{main thm} Let $\{X_t|t\in \triangle(\epsilon)\}$ be an analytic
family of compact complex $n$-dimensional manifolds. If $X_0$ admits
a balanced metric and for small $t\not=0$, $X_t$ satisfies the
$(n-1,n)$-th weak $\partial\bar\partial$-lemma, then for
sufficiently small $t$, $X_t$ also admits a balanced metric.
\end{theo}

So when $X_t$ for $t\not=0$ satisfies the
$\partial\bar\partial$-lemma, then the existence of balanced metrics
is preserved under small deformations. This result at  present is
mildly stronger than Theorem \ref{Wu2} since we don't know whether
the $\partial\bar\partial$-lemma property is also closed under small
deformations.

We should check that the small deformation of the Iwasawa manifold
 mentioned above does not satisfy  the condition in Theorem \ref{main thm}.
Actually if we use the notations in \cite{AB1}, we find that on
$X_t$ when $t\not=0$,
\begin{equation*}
\bar\partial_t\bigl(\phi_{1,t}\wedge
\bar\phi_{1,t}\wedge\phi_{3,t}\wedge
\bar\phi_{3,t}\bigr)=t\partial_t\bigl(\phi_{3,t}\wedge
\bar\phi_{1,t}\wedge \bar\phi_{2,t}\wedge\bar\phi_{3,t}\bigr)
\end{equation*}
can not be written as a $\partial_t\bar\partial_t$-exact form. Here
$d$ is decomposed as $d=\partial_t+\bar\partial_t$ according to the
complex structure of $X_t$.
\\

  An application of the main
result is
\begin{coro}Let $\{X_t|t\in \triangle(\epsilon)\}$ be an analytic
family of compact complex $n$-dimensional manifolds.  Suppose $X_0$
admits a balanced metric and the Dolbeault cohomology group
$H^{2,0}(X_t,\mathbb C)=0$ for small $t\not =0$. Then there exists a
balanced metric on $X_t$ for sufficiently small $t$.
\end{coro}
\begin{proof}
By  Serre duality, we have  $H^{n-2,n}(X_t,\mathbb
C)=H^{2,0}(X_t,\mathbb C)=0$. Now let $\varphi_t$ be a real
$(n-1,n-1)$-form on $X_t$ such that $\bar\partial_t\varphi_t$ is a
$\partial_t$-exact form, i.e., there exist an $(n-2,n)$-form
$\eta_t$ such that $\bar\partial_t\varphi_t=\partial_t\eta_t$. Since
$\bar\partial_t\eta_t=0$ and $H^{n-2,n}(X_t,\mathbb C)=0$, there
exists an $(n-2,n-1)$-form $\psi_t$ such that
$\eta_t=i\bar\partial_t\psi_t$. Therefore
$\bar\partial_t\varphi_t=i\partial_t\bar\partial_t\psi_t$. So $X_t$
satisfies the $(n-1,n)$-th weak $\partial\bar\partial$-lemma.
\end{proof}

\begin{coro}
Let $\{X_t|t\in \triangle(\epsilon)\}$ be an analytic family of
compact complex $n$-dimensional manifolds. If $X_0$ admits a
balanced metric and $H^{2,0}(X,\mathbb C)=0$, then $X_t$ also admits
a balanced metric for small enough $t$.
\end{coro}
\begin{proof}
Since the function $h^{p,q}(t)=\dim H^{p,q}(X_t,\mathbb C)$ is upper
semicontinuous in $t$, $h^{2,0}(t)\leq h^{2,0}(0)=0$ for small $t$. So
$H^{2,0}(X_t,\mathbb C)=0$. Then the corollary follows.
\end{proof}

It is well-known that the twistor space $Z$ associated to a compact
self-dual four manifold is a complex manifold and the natural metric
on it is a balanced metric (c.f. \cite{Mic,Gau}). Moreover, M.
Eastwood and M. Singer in \cite{ES} observed  $H^{2,0}(Z,\mathbb
C)=0$. So above corollary implies
\begin{coro}
Let $Z$ be the twistor space associated to a compact self-dual  four
manifold. Then any small deformation of $Z$  admits a balanced
metric.
\end{coro}

Up to now, the global deformation stability of existence of balanced
metrics is unclear (c.f. \cite{Yau,Po}). We shall use above corollary to study this question.\\

\noindent {\bf Acknowledgment:} The idea of the proof of the main
result is from \cite{FLY}. The authors would like to thank Professor
J. Li for useful discussions. Fu is partially supported by NSFC
grants 11025103 and 10831008. Yau is partially supported by NSF
grant DMS-0804454.

\section{The proof of the main theorem}
Let $X_0$ be a compact complex manifold with a balanced metric
$\omega$. Let $\pi:\mathfrak{X}\to \bigtriangleup(\epsilon)$ be a
small deformation of $X_0$. That is, $\pi:\mathfrak{X}\to
\bigtriangleup(\epsilon)$ is a  holomorphic map with maximal rank so
that $\pi$ is proper and each fiber $X_t=\pi^{-1}(t)$ has the
structure of a complex manifold which varies analytically with $t$
(c.f. \cite{MK}). Let $\phi_t:X_t\to X_0$ be a diffeomorphism which
varies smoothly with $t$ and $\phi_0$ is the identity map. Since
$\omega^{n-1}$ is a $d$-closed real $(n-1,n-1)$-form on $X$,
$\Omega_t=\phi_t^\ast\omega^{n-1}$ is a $d$-closed real
$(2n-2)$-form on $X_t$. We decompose $\Omega_t$ as
\begin{equation*}
\Omega_t=\Omega_t^{n-2,n}+\Omega_t^{n-1,n-1}+\Omega_t^{n,n-2}.
\end{equation*}
Then we have the following facts:
\begin{enumerate}
 \item  $\Omega_t^{n-1,n-1}$ is real and approaches $\omega^{n-1}$ as $t\to 0$;\\
\item $\overline{\Omega_t^{n-2,n}}=\Omega_t^{n,n-2}$ and
$\Omega_t^{n-2,n}$ approaches zero as $t\to 0$;\\
\item Since $\Omega_t$ is $d$-closed, we have
\begin{eqnarray*}
\begin{aligned}
\bar\partial_t\Omega_t^{n-1,n-1}+\partial_t\Omega_t^{n-2,n}=0.
\end{aligned}
\end{eqnarray*}
\end{enumerate}

According to (1), $\Omega_t^{n-1,n-1}$ is strictly positive definite
for sufficiently small $t$. Then there exists a hermitian metric
$\omega_t$ on $X_t$  such that $\omega_t^{n-1}=\Omega_t^{n-1,n-1}$.
This observation is due to Michelsohn \cite{Mic}. Clearly
$\omega_0=\omega$ and $\omega_t$ approaches $\omega$ smoothly and
uniformly as $t\to 0$. We use this metric $\omega_t$ as the
background metric on $X_t$.

If  $X_t$ satisfies the $(n-1,n)$-th weak
$\partial\bar\partial$-lemma, then (3) implies that there exists an
$(n-2,n-1)$-form $\Psi_t$ on $X_t$ such that
\begin{equation}\label{ddbe}
i\partial_t\bar\partial_t\Psi_t=\bar\partial_t\Omega_t^{n-1,n-1}=-\partial_t\Omega_t^{n-2,n}.
\end{equation}
We can choose $\Psi_t$ such that
\begin{equation*}
\Psi_t\perp_{\omega_t}\ker(i\partial_t\bar\partial_t).
\end{equation*}
We let
$$\tilde\Omega_t=\Omega_t^{n-1,n-1}+i\partial_t\Psi_t-i\bar\partial\bar\Psi_t.$$
 Then $\tilde\Omega_t$ is a $d$-closed real $(n-1,n-1)$-form. If we can prove $\tilde\Omega_t$ is strictly
 positive definite for sufficiently small $t$, then there exists a hermitian metric $\tilde\omega_t$ such
 that $\tilde\omega_t^{n-1}=\tilde\Omega_t$ and
 $d\tilde\omega_t^{n-1}=0$. That is, we have gotten a balanced
 metric $\tilde\omega_t$ on $X_t$ as desired.\\

So we only need to prove the positivity of $\tilde\Omega_t$. To this
end we introduce the Kodaira-Spencer operator $E_t$ on
$\Lambda^{n-1,n}(X_t)$, i.e., on the set of all $(n-1,n)$-forms on
$X_t$
\begin{equation*}
E_t=\partial_t\bar\partial_t\bar\partial_t^\ast\partial_t^\ast
+\partial_t^\ast\bar\partial_t\bar\partial_t^\ast\partial_t
+\partial_t^\ast\partial_t,
\end{equation*}
where the Hodge-star operator $\ast$, which we have dropped the
subscript $t$, is defined by the metric $\omega_t$. We consider the
equation
\begin{equation}\label{basiceq}
E_t(\gamma_t)=-\partial_t\Omega_t^{n-2,n}(=\bar\partial_t\Omega_t^{n-1,n-1}).
\end{equation}
We first check that $\partial_t\Omega_t^{n-2,n}\perp_{\omega_t} \ker
E_t$. In \cite{KS}, Kodaira and Spencer proved that $E_t$ is
self-adjoin, strongly elliptic of order 4, and a form
$\alpha_t\in\ker E_t$ if and only if \begin{equation*}
\partial_t\alpha_t=0 \ \ \ \ \textup{and}\ \ \ \
\bar\partial^\ast_t\partial^\ast_t\alpha_t=0.
\end{equation*}
Then by (\ref{ddbe}), for any $\alpha_t\in \ker E_t$, we have
\begin{equation*}
(-\partial_t\Omega_t^{n-2,n},\alpha_t)=(i\partial_t\bar\partial_t\Psi_t,
\alpha_t)=(i\Psi_t,\bar\partial^\ast_t\partial^\ast_t\alpha_t)=0.
\end{equation*}
Therefore, by the theory of elliptic operators, there exists a
unique smooth solution $\gamma_t$ of equation (\ref{basiceq}) such
that $\gamma_t\perp_{\omega_t} \ker E_t$.

As in \cite{FLY}, we can  prove
\begin{equation}\label{ccc}
\partial_t\gamma_t=0\ \ \ \text{and}\ \ \
i\Psi_t=\bar\partial_t^\ast\partial_t^\ast\gamma_t.
\end{equation}
For readers'convenience, we prove these two identities here.  From
(\ref{ddbe}) and (\ref{basiceq}), we get
$E_t(\gamma_t)-i\partial_t\bar\partial_t\Psi_t=0$, which, from the
definition of the operator $E_t$, is equivalent to
\begin{equation*}
\partial_t\bar\partial_t(\bar\partial_t^\ast
\partial_t^\ast\gamma_t-i\Psi_t)
+\partial_t^\ast(\bar\partial_t\bar\partial_t^\ast+1)
\partial_t\gamma_t=0.
\end{equation*}
By taking  the $L^2$-norm of the left hand side, we get
\begin{equation}\label{aaa}
\partial_t\bar\partial_t(\bar\partial_t^\ast
\partial_t^\ast\gamma_t-i\Psi_t)
=0\ \ \ \text{and}\ \ \
\partial_t^\ast(\bar\partial_t\bar
\partial_t^\ast+1)\partial_t\gamma_t=0.
\end{equation}
On the other hand, for any $\phi\in\ker
(i\partial_t\bar\partial_t)$, we have
$$(\bar\partial_t^\ast\partial_t^\ast\gamma_t,\phi) =(\gamma_
t,\partial_t\bar\partial_t\phi)=0.$$ Since
$\Psi_t\bot_{\omega_t}\ker(i\partial_t\bar\partial_t)$,
\begin{equation}\label{bbb}
(\bar\partial_t^\ast\partial_t^\ast\gamma_t-i\Psi_t)\bot_{\omega_t}
\ker (i\partial_t\bar\partial_t).
\end{equation}
Combining (\ref{aaa}) with (\ref{bbb}), we obtain
$\bar\partial_t^\ast\partial_t^\ast\gamma_t-i\Psi_t=0$, which is the
first identity in (\ref{ccc}). The second in (\ref{ccc}) follows
from the second equality of (\ref{aaa}), since
\begin{equation*} 0=\int_{X_t}\langle \partial_t^\ast(\bar\partial_t\bar\partial_t^\ast+1)
\partial_t\gamma_
t,\gamma_t\rangle=\int_{X_t}(|
\bar\partial_t^\ast\partial_t\gamma_t| ^2+|
\partial_t\gamma_t| ^2).
\end{equation*}

Now we can estimate $\parallel
i\partial_t\Psi_t\parallel_{C^0(\omega_t)}$ by elliptic estimates.
Since the background metric $\omega_t$ varies smoothly with $t$ and
approaches $\omega_0$ as $t\to 0$,  we have a uniform constant $C$
such that
\begin{equation}
\parallel i\partial_t\Psi_t\parallel_{C^0(\omega_t)}\leq C\parallel
\gamma_t\parallel_{C^3(\omega_t)}\leq
C\parallel\partial_t\Omega_t^{n-2,n}\parallel_{C^{0,\alpha}(\omega_t)}
\end{equation}
for some $0< \alpha <1$. Since also $\Omega_t^{n-2,n}$ varies
smoothly with $t$ and approaches zero as $t\to 0$  and the complex
structure on $X_t$ varies analytically with $t$,
$\parallel\partial_t\Omega_t^{n-2,n}\parallel_{C^{0,\alpha}(\omega_t)}$
approaches zero uniformly as $t\to 0$. Thus $\parallel
i\partial_t\Psi_t\parallel_{C^{0}(\omega_t)}$ approaches zero as
$t\to 0$. Therefore $\tilde\Omega_t$ is strictly positive definite
when $t$ is small enough.

\end{document}